\documentclass[10pt]{amsart}

\usepackage{amssymb,latexsym, mathtools,tikz}
\usepackage{enumerate}
\usepackage{graphicx}
\usepackage{float}
\usepackage{placeins}
\usepackage{mdframed}
\usepackage{amssymb}
\usepackage{esint}
\usepackage{cool}
\usepackage[all,cmtip]{xy}
\usepackage{mathtools}
\usepackage{amstext} 
\usepackage{array}   
\usepackage[shortlabels]{enumitem}
\usepackage{ytableau}

\newcolumntype{L}{>{$}l<{$}} 
\newcolumntype{C}{>{$}c<{$}}
\newtheorem{theorem}{Theorem}[section]
\newtheorem{lemma}[theorem]{Lemma}
\newtheorem{cor}[theorem]{Corollary}
\newtheorem{prop}[theorem]{Proposition}
\newtheorem{setup}[theorem]{Setup}
\theoremstyle{definition}
\newtheorem{definition}[theorem]{Definition}
\newtheorem{example}[theorem]{Example}

\newtheorem{obs}[theorem]{Observation}
\newtheorem{notation}[theorem]{Notation}

\newtheorem{construction}[theorem]{Construction}

\theoremstyle{remark}
\newtheorem{remark}[theorem]{Remark}

\newtheorem{the context}[theorem]{The Context}
\newtheorem{question}[theorem]{Question}
\numberwithin{equation}{theorem}
\numberwithin{equation}{section}

\newcommand{\rank}{\operatorname{rank}}

\newcommand{\Ker}{\operatorname{Ker}}

\newcommand{\ideal}[1]{\mathfrak{#1}}
\newcommand{\m}{\ideal{m}}

\newcommand{\supp}{\operatorname{Supp}}

\newcommand{\bbz}{\mathbb{Z}}

\renewcommand{\geq}{\geqslant}
\renewcommand{\leq}{\leqslant}
\renewcommand{\ker}{\Ker}
\renewcommand{\hom}{\Hom}

\newcommand{\Hom}{\operatorname{Hom}}

\newcommand{\maps}[5]{\xymatrix{#1 \ar[r]^-{#3} & #2 \\
#4 \ar@{|->}[r] & #5 \\}}

\newcommand{\set}{\operatorname{set}}

\newcommand{\lin}{\operatorname{lin}}

\begin{document}
\title{Iterated Mapping Cones for Strongly Koszul Algebras}

\author{Keller VandeBogert}
\address{University of Notre Dame}
\date{\today}

\maketitle

\begin{abstract}
    In this paper we extend the well-known iterated mapping cone procedure to monomial ideals in strongly Koszul algebras. We study properties of ideals generated by monomials in commutative Koszul algebras and show that the linear strand of ideals generated by linear forms is obtained as a subcomplex of the Priddy complex. In the case of strongly Koszul algebras, this shows that the minimal free resolution of a monomial ideal admitting linear quotients is obtained as an iterated mapping cone, immediately extending results for such ideals in polynomial rings to strongly Koszul algebras. We then consider monomial ideals admitting a so-called regular ordering, a generalization of regular decomposition functions, and show that the comparison maps in the iterated mapping cone construction can be computed explicitly. In particular, this gives a closed form for the minimal free resolution of monomial ideals admitting a regular ordering over strongly Koszul algebras.
\end{abstract}

\section{Introduction}

A $k$-algebra $A$ is \emph{Koszul} if its residue field $k$ has a linear minimal free resolution over $A$. Koszul algebras are one method of generalizing standard graded polynomial rings and make their appearance in a wide range of seemingly disconnected settings. Topologically, Koszul duality for quadratic algebras can be used to translate between facts about equivariant and standard cohomology (see \cite{goresky1997equivariant}). Representation theoretically, Koszul algebras are, in the words of Beilinson, Ginzburg, and Soergel, ``as close to semisimple as a $\bbz$-graded ring can possibly be" and arise in the context of mixed complexes/Hodge modules (see \cite{beilinson1996koszul}). In the setting of number theory, Positselski (see \cite{positselski2014galois}) has shown that certain classes of Milnor rings are Koszul algebras and relates the Milnor-Bloch-Kato conjecture to Koszulness of certain quotient rings. Fr\"oberg gives an overview in \cite{froberg1999koszul} of conditions equivalent to Koszulness along with many examples in the literature of Koszul algebras, including extremal Gorenstein rings and Veronese/Segre algebras. For a more complete exposition/overview of Koszul algebras, see \cite{polishchuk2005quadratic}.

For non-regular Koszul algebras, the minimal free resolution of the residue field is necessarily infinite, and this resolution can be explicitly described by the Priddy complex, introduced by Stewart Priddy in \cite{priddy1970koszul}. One interesting feature of this resolution is the association of the algebra $A$ to its quadratic dual, $A^!$. In the ``classical" case of a polynomial ring, $A$ is the symmetric algebra on some vector space and its quadratic dual is the exterior algebra; in particular the Priddy complex recovers the well-known Koszul complex. This leads to the idea that one can try to generalize constructions involving symmetric and exterior algebras by replacing these objects with Koszul algebras and their respective quadratic duals.

This is the approach taken by the authors in \cite{faber2020canonical} (and indeed was the original motivation for this paper), where the $L$-complex construction of Buchsbaum and Eisenbud (see \cite{buchsbaum1975generic}) was extended to Koszul algebras by replacing certain Schur modules associated to hook partitions with the image of the natural maps
$$A \otimes_k (A^!)^*_{i+1} \otimes_k A_{d-1}  \to A \otimes_k (A^!)^*_i \otimes_k A_d$$
induced by multiplication by the trace element. It is shown that the complexes induced by these building blocks constitute a minimal free resolution of powers of the maximal ideal of $A$, a direct generalization of the original case considered in \cite{buchsbaum1975generic}.

Powers of the maximal ideal are, in particular, monomial ideals, with the minimal generating set given by the classes of all monomials of a given degree (one has to be careful about what they mean by ``monomial ideal" in a general Koszul algebra; see Remark \ref{rk:MonIdeals}). Considering the combinatorial nature of many naturally occurring Koszul algebras (for instance, Hibi rings on posets, Stanley-Reisner rings of graphs, or (sometimes) Orlik-Solomon algebras associated to matroids), one is tempted to ask: what techniques relating to monomial ideals in polynomial rings generalize to arbitrary Koszul algebras?

In this paper, we begin to answer this question from a homological perspective by generalizing the well-known iterated mapping cone procedure for ideals with linear quotients (see \cite{herzog2002resolutions}) to \emph{strongly} Koszul algebras. The iterated mapping cone procedure spawns from the observation that the short exact sequence
$$0 \to \frac{A}{(I:m)} \xrightarrow[]{m} \frac{A}{I} \xrightarrow[]{} \frac{A}{I +(m)} \to 0$$
can be used to inductively compute a free resolution of any quotient ring; in the case that the ideal has linear quotients, this procedure will yield the \emph{minimal} free resolution. This recovers the well-known (squarefree) Eliahou-Kervaire resolution for (squarefree) stable ideals (see \cite{eliahou1990minimal} and \cite{aramova1998squarefree}), and extends to many more classes of monomial ideals. One hitch to this approach for general Koszul algebras is the fact that an ideal generated by a subset of variables may not have linear resolution (see Example \ref{ex:badKoszul}); this is remedied by restricting to \emph{strongly} Koszul algebras (see Definition \ref{def:strongKoszul}), originally introduced by Herzog, Hibi, and Restuccia in \cite{herzog2000strongly}.  

This paper is organized as follows. In Section \ref{sec:background} we introduce/recall some results on linear strands and Koszul algebras. Most importantly, we recall the construction of the quadratic dual of a quadratic algebra and use this to construct the Priddy complex. In Section \ref{sec:itMapCones}, we begin the process of extending results on monomial ideals in polynomial rings to Koszul algebras. In particular, we construct the linear strand of any ideal generated by linear forms as a subcomplex of the Priddy complex (see Proposition \ref{prop:varslinRes}); in the case of a strongly Koszul algebra, this yields the minimal free resolution explicitly. We then generalize the idea of a \emph{decomposition function} to general strongly Koszul algebras and show that the minimal free resolutions of ideals admitting linear quotients can be constructed as an iterated mapping cone. 

In Section \ref{sec:theMFRforRegular}, we consider cases for which the induced comparison maps in the iterated mapping cone procedure can be constructed explicitly. This leads to the definition of monomial ideals admitting a \emph{regular ordering} (see Definition \ref{def:regulrityDef}), a direct generalization of \emph{regular} decomposition functions as in \cite{herzog2002resolutions}. This culminates in Theorem \ref{thm:theMFR}, whose proof includes the original Theorem $1.12$ of \cite{herzog2002resolutions}, and shows that the minimal free resolution of any monomial ideal admitting a regular ordering in a strongly Koszul algebra can be computed explicitly in a manner that directly generalizes the case for polynomial rings. We conclude with further questions on different choices of decomposition for computing the comparison maps in the iterated mapping cone procedure and additional combinatorial structure on the resolution of Theorem \ref{thm:theMFR} arising from infinite-dimensional cell complexes.

\section{Koszul Algebras and The Priddy Complex}\label{sec:background}

In this section, we introduce some necessary background on linear strands and Koszul algebras that will be needed for the remainder of the paper. We record two results of \cite{herzog2015linear} that are useful for proving that a candidate complex does indeed arise as the linear strand of a module. After defining Kozul algebras, we introduce some machinery for defining the so-called Priddy complex. We conclude with the well-known connection between Koszul algebras and acyclicity of the associated Priddy complex, originally proved in \cite{priddy1970koszul}.

Throughout the paper, all complexes are assumed to be concentrated only in nonnegative degrees. All $k$-algebras will be assumed to be finitely generated and endowed with the standard grading (that is, all variables have degree $1$), and denoted by $A$. 

\begin{definition}\label{def:linStrand}
Let $F_\bullet$ be a minimal graded $A$-free complex with $F_0$ having initial degree $d$. Then the \emph{linear strand} of $F_\bullet$, denoted $F_\bullet^{\lin}$, is the complex obtained by restricting $d_i^F$ to the subcomplex generated by components $(F_i)_{d+i}$ for each $i \geq 1$.
\end{definition}

\begin{remark}
Observe that the minimality assumption in Definition \ref{def:linStrand} ensures that the linear strand is well defined. Choosing bases, the linear strand can be obtained by restricting to the columns where only linear entries occur in the matrix representation of each differential.
\end{remark}

\begin{theorem}[\cite{herzog2015linear}, Theorem 1.1]\label{thm:linstrandequiv}
Let $G_\bullet$ be a linear complex of free $A$-modules with initial degree $n$. Then the following are equivalent:
\begin{enumerate}
    \item The complex $G_\bullet$ is the linear strand of a finitely generated $A$-module with initial degree $n$.
    \item The homology $H_i (G_\bullet)_{i+n+j} = 0$ for all $i>0$ and $j=0, \ 1$.
\end{enumerate}
\end{theorem}

The following Proposition shows that linear strands behave in a manner that is quite similar to minimal free resolutions.

\begin{prop}[\cite{herzog2015linear}, Corollary 1.2]\label{prop:linstrandcor}
Let $G_\bullet$ be a linear complex of free $A$-modules with initial degree $n$ such that $H_i (G_\bullet)_{i+n+j} = 0$ for all $i>0$, $j=0, \ 1$. 

Let $N$ be a finitely generated $A$-module with minimal graded free resolution $F_\bullet$. Assume that there exist isomorphisms making the following diagram commute:
$$\xymatrix{G_1 \ar[d]_-{\sim} \ar[r] & G_0 \ar[d]^-{\sim} \\
F_1^{\textrm{lin}} \ar[r] & F_0^{\textrm{lin}}. \\}$$
Then $G_\bullet \cong F_\bullet^\textrm{lin}$.
\end{prop}

\begin{remark}
In the original formulation of Theorem \ref{thm:linstrandequiv} and Proposition \ref{prop:linstrandcor}, it was assumed that the complex $G_\bullet$ had finite length. However, the proofs extend with no effort to the infinite case.
\end{remark}

Next, we introduce our main objects of study for the rest of the paper.

\begin{definition}
A $k$-algebra $A$ is \emph{Koszul} if the residue field $k$ has a linear graded minimal free resolution over $A$.
\end{definition}

\begin{example}
Let $A = k[x_1 , \dots , x_n]$. Then the ideal of variables forms a regular sequence, so the residue field is resolved by the Koszul complex on $x_1 , \dots , x_n$, which is a linear complex.  
\end{example}

The following will be needed for stating Definition \ref{def:PriddyComplex}.

\begin{definition}
Let $A$ be a graded $k$-algebra. The algebra $A$ is \emph{quadratic} if $A = T(V) / Q$, where $V$ is a $k$-vector space, $T(V)$ is the tensor algebra, and $Q$ is a quadratic ideal of $T(V)$. 

Given a quadratic $k$-algebra $A$, the \emph{quadratic dual} is the algebra
$$A^! := \frac{T(V^*)}{Q^\perp},$$
with $V^* := \hom_k (V,k)$, and $Q^\perp$ is the quadratic ideal generated by the orthogonal complement to $Q_2$ with respect to the natural pairing
\begingroup\allowdisplaybreaks
\begin{align*}
    (V \otimes V) \otimes (V^* \otimes V^*) &\to k, \\
    (v_1 \otimes v_2) \otimes (u_1 \otimes u_2) &\mapsto u_1(v_1) \cdot u_2 (v_2). \\
\end{align*}
\endgroup
\end{definition}

\begin{example}
If $A = S(V)$ is the symmetric algebra on some vector space $V$, then $A^! = \bigwedge V^*$ is the exterior algebra on $V^*$.
\end{example}

\begin{notation}
If $x_1 , \dots , x_n$ are the generators of the $k$-algebra $A$, then the notation $x_1^*, \dots , x_n^*$ will denote the dual generators of $A^!$. 
\end{notation}

\begin{definition}\label{def:PriddyComplex}
Let $A$ be a quadratic algebra. The \emph{Priddy complex} is the complex $P_\bullet^A$ with
$$P_i^A := A \otimes_k (A^!)_i^*$$
and differential $P_i^A \to P_{i-1}^A$ induced by multiplication by the trace element, $\sum_i x_i \otimes x_i^*$.
\end{definition}

The following theorem was proved by Priddy in \cite{priddy1970koszul}; other proofs may be found in, for instance, \cite{polishchuk2005quadratic}.

\begin{theorem}
Let $A$ be a quadratic $k$-algebra. Then $A$ is Koszul if and only if the Priddy complex $P_\bullet^A$ is acyclic.
\end{theorem}

\section{Iterated Mapping Cones for Strongly Koszul Algebras}\label{sec:itMapCones}

In this section, we turn our attention to the iterated mapping cone construction for Koszul algebras. As it turns out, this construction may not yield a minimal free resolution for arbitrary Koszul algebras (see Example \ref{ex:badKoszul}), but everything works out nicely if we restrict to the slightly smaller class known as \emph{strongly Koszul algebras}. These were introduced originally by Herzog, Hibi, and Restuccia, and have the convenient property that any ideal generated by variables has linear resolution. We also show that for an arbitrary Koszul algebra, the linear strand of any ideal generated in degree $1$ is obtained as a subcomplex of the Priddy complex. In the case of strongly Koszul algebras, this subcomplex will be the minimal free resolution. 

With these results established, we obtain a more general version of Lemma $1.5$ of \cite{herzog2002resolutions}. As an immediate consequence, ideals with linear quotients in strongly Koszul algebras satisfy a nearly identical set of properties to those in the polynomial ring case. For the remainder of the paper, $A := k [ x_1 , \dots , x_n ] / I$ will denote a commutative Koszul algebra.

\begin{definition}\label{def:strongKoszul}
A Koszul algebra $A$ is \emph{strongly Koszul} if there exists a basis $X$ for $A_1$ such that for every $Y \subset X$ and every $x \in X \backslash Y$, there exists a subset $Z \subset X$ such that $\big( (Y) : x \big) = (Z)$.

A Koszul algebra $A$ is \emph{universally Koszul} if every ideal generated by linear forms has a linear resolution.
\end{definition}

\begin{remark}
The above definition of strongly Koszul is a definition introduced by Conca, De Negri, and Rossi in \cite[Definition 3.11]{conca2013koszul}, and is slightly stronger than the original definition used by Herzog, Hibi, and Restuccia \cite{herzog2000strongly}. Throughout the rest of the paper, we will assume by convention that if $A$ is a strongly Koszul algebra, then the presentation $A = k[x_1 , \dots , x_n]/I$ has been chosen in such a way that $X = \{ x_1 , \dots , x_n \}$.
\end{remark}

The following Lemma is an essential component of what makes the iterated mapping cone construction work well for strongly Koszul algebras.
\begin{lemma}[\cite{herzog2000strongly}]\label{lem:linearVars}
Let $A$ be a strongly Koszul algebra. Then every ideal generated by a subset of the variables has a linear minimal free resolution.
\end{lemma}

It is important to note that Lemma \ref{lem:linearVars} does not hold for general Koszul algebras. The following example was coomunicated to the author by Hailong Dao, and is a well-known example due to Conca of a Koszul algebra that is \emph{not} LG-quadratic (see \cite[Example 1.20]{conca2014koszul}).

\begin{example}\label{ex:badKoszul}
Let
$$A := \frac{k[a,b,c,d]}{(ac,ad,ab-bd,a^2+bc,b^2)}.$$
Then $A$ is a Koszul algebra, but the ideal $(b)$ does not have linear resolution since the relation $c^2 \cdot b = 0$ is minimal.
\end{example}

The following is an exercise in linear algebra which gives a straightforward relation between generators of the Koszul algebra and its quadratic dual.

\begin{obs}\label{obs:A!relations}
Let $A$ be a commutative Koszul algebra and let $\{ x_s x_t \mid (s,t) \in S \}$ be a basis for $A_2$. For all $(u,v) \notin S$, there exist coefficients $f^{u,v}_{s,t} \in k$ such that
$$x_u x_v = \sum_{(s,t) \in S} f^{u,v}_{s,t} x_s x_t.$$
Given this notation, $A^!_2$ has basis $\{ x^*_u x^*_v \mid (u,v) \notin S \}$ with relations
$$x^*_s x^*_t = - \sum_{(u,v) \notin S} f^{u,v}_{s,t} x^*_u x^*_v.$$
\end{obs}

In the following construction, we introduce the subcomplex of the Priddy complex which will serve as the linear strand/minimal free resolution of the ideals of interest.

\begin{construction}\label{cons:subPriddy}
Let $A$ be a Koszul algebra and $(\ell_1 , \dots , \ell_k)$ an ideal minimally generated by linear forms, and let $L := \sum_{i \notin E} A^! x_i^*$ be the left ideal generated by the variables with indices not appearing in $E$. Define
$$B := A^!/L.$$
Observe that $B^*$ inherits the structure of a right $A^!$-module via the action $(\phi \cdot a) (b) := \phi (a b)$, for $\phi \in B^*$, $a \in A^!$, and $b \in B$. Consider the induced complex
$$\cdots \to A \otimes B^*_d \to A \otimes B^*_{d-1} \to \cdots \to A \otimes B^*_1 \to A \to 0,$$
where each differential is right multiplication by the trace element $\sum_{i =1}^n x_i \otimes x_{i}^*$. 
\end{construction}

\begin{prop}\label{prop:subPriddyCx}
Let $A$ be a Koszul algebra and $A \otimes B^*_\bullet$ denote the complex of Construction \ref{cons:subPriddy}. Then
$$(H_i (A \otimes B^*_\bullet))_{i+j} = 0 \quad \textrm{for} \ i>0 \ \textrm{and} \ j=0, 1.$$
\end{prop}

\begin{proof}
Making an appropriate change of variable, it is of no loss of generality to assume that each $\ell_j = x_{i_j}$ is a variable. Moreover, this is compatible with the Priddy differential since this differential is induced by multiplication by the image of $1$ under the isomorphism $\hom (V,V) \cong V \otimes V^*$, which is independent of the choice of basis for $V$.

Notice that $H_i(A \otimes B_\bullet^*)_{i} = 0$ trivially, since there is a natural inclusion $A \otimes B_\bullet^* \hookrightarrow A \otimes (A^!)^*_\bullet$ (obtained by dualizing the natural surjection $A^! \to B$), and there are no degree $i$ cycles in homological degree $i$ in the Priddy complex. 

Next, let $ y = \sum_j \lambda_j x_j \otimes f_j$ be a cycle of degree $i+1$ in homological degree $i$. Since the Priddy complex is acyclic, $y = d(\lambda \otimes f)$ for some $f \in (A^!)^*_{i+1}$. It remains to show that if $ f \cdot x_j^* \in B^*_{i}$ for all $j$, then $f \in B^*_{i+1}$. By definition, there is a short exact sequence
$$0 \to B^* \to (A^!)^* \to L^* \to 0,$$
whence $f \in B^*$ if and only if $f(L) = 0$. Let $h \in L$ be any homogeneous element of degree $r$; if $r = 1$, then by assumption $f \cdot x_s^* = 0$ for every $s \notin E$. If $r>1$, then $h = \sum_{s \notin E} h_s x_s^*$; since $\deg (h_s) \geq 1$, one can write $h_s = \sum_{j=1}^k h_{j,s} x_s^*$ and hence 
$$ \sum_{j=1}^k x_j^* \Big( \sum_{s \notin E} h_{j,s} x_s^* \Big).$$
Since $f x_j^* = 0$ and $\sum_{s \notin E} h_{j,s} x_s^* \in L$ for each $1 \leq j \leq k$, one finds that $f(h) = 0$ and hence $f \in B^*_i$ as desired.
\end{proof}

\begin{prop}\label{prop:varslinRes}
Let $A$ be a Koszul algebra. Then the ideal $(\ell_1 , \dots , \ell_k)$ has linear strand given by the complex of Construction \ref{cons:subPriddy}. If $A$ is universally Koszul, then this complex is the minimal free resolution of $A/(\ell_1 , \dots , \ell_k)$.
\end{prop}

\begin{proof}
Assume without loss of generality that each $\ell_i$ is a variable. Let $L := (\ell_1 , \dots , \ell_k)$ and $F_\bullet$ denote the minimal free resolution of $L$. In view of Propositions \ref{prop:subPriddyCx} and \ref{prop:linstrandcor}, it suffices to show that there is an isomorphism $F_1^{\lin} \to F_0^{\lin} \cong A \otimes B_2^* \to A \otimes B_1^*$. From the natural inclusion $A \otimes B_\bullet^* \hookrightarrow A \otimes (A^!)^*_\bullet$ it is immediate that the map $A \otimes B_1^* \to A$ is a minimal presentation of $A/L$ and that multiplication by the trace element induces a map $A \otimes B_2^* \to Z_2$, where $Z := \ker (A \otimes B_1^* \to L)$.   

Let $\sum_i a_i \ell_i =0$ be any linear relation on the generators of $L$. This relation is evidently the image of $\sum_i a_i \otimes \ell_i$ under the Priddy differential. Moreover, by Proposition \ref{prop:subPriddyCx}, $\sum_i a_i \otimes \ell_i$ must lie in the image of $A \otimes B_2^* \to A \otimes B_1^*$, whence the result follows.
\end{proof}

With Proposition \ref{prop:varslinRes} established, we introduce some necessary notation for monomial ideals in strongly Koszul algebras.

\begin{notation}\label{not:MonIdeals}
Let $A$ be a strongly Koszul algebra and $J = (m_1 , \dots , m_r)$ be a monomial ideal of $A$. The following notation will be used:
$$M(J) := \textrm{the set of all monomials of} \ J,$$
$$G(J) := \textrm{the unique minimal generating set of} \ J, \ \textrm{consisting of monic monomials.}$$
With the generators of $J$ ordered as above, the notation $J_i$ will denote the ideal $(m_1 , \dots , m_i)$, where by convention $J_0 = (0)$. Given any monomial $m \in A$, define
$$\supp (m) := \{ x_i \mid m \in (x_i) \}.$$
Given any monomial ideal $J$, define
$$\supp (J) := \bigcup_{m \in G(J)} \supp (m).$$
\end{notation}

\begin{remark}\label{rk:MonIdeals}
It will be understood that an ideal will be referred to as a monomial ideal if, after a $k$-basis of $A$ has been chosen, all minimal generators can be written as monomials with respect to this basis. It is important to note that ideals that are monomial ideals with respect to one choice of basis may \emph{not} be monomial ideals with respect to a different choice.

As a very simple example of this, the ideal $(xy) \subset \frac{k[x,y,z]}{(xy-yz-xz)} =: A$ is a monomial ideal as long as the element $xy$ is considered as part of a $k$-basis for $A$. This is important since otherwise the notation $G(J)$ as in Notation \ref{not:MonIdeals} may not be well defined.
\end{remark}

\begin{definition}
Let $A$ be a strongly Koszul algebra and $J = (m_1 , \dots , m_r)$ a monomial ideal of $A$. Then $J$ admits \emph{linear quotients} with respect to the given ordering if for all $2 \leq i \leq n$, the ideal $(J_{i-1} : m_i)$ is generated by a subset of the variables.
\end{definition}

\begin{obs}\label{obs:varsColon}
Let $A$ be a strongly Koszul algebra. Then for any monomial $m \in A$, the ideal $(0: m)$ is generated by a subset of variables. In particular, $(0 : m)$ has linear resolution.
\end{obs}

\begin{proof}
Proceed by induction on the degree of $m$. In the case that $\deg m = 1$, $m$ is a variable and the statement holds by the assumption that $A$ is strongly Koszul. Now, let $\deg (m) >1$ and write $m = m' x_i$ for some $x_i \in \supp (m)$ with $\deg (m') < \deg(m)$. By induction, $(0: m')$ is generated by a subset of variables and, since $(0: m) = ((0: m') : x_i)$, the ideal $(0: m)$ is also generated by variables by the definition of strongly Koszul.


\end{proof}

The following proposition is the generalized version of ``decomposition functions" as in \cite{herzog2002resolutions}. In a general Koszul algebra, if a monomial $m$ is contained in some other monomial ideal, this does not imply that $m$ must be a multiple of one of the minimal generators. For example, let $J = (xy,yz) \subset k[x,y,z]/(xy+xz+yz)$. By construction, $xz \in J$, but $xz$ is not a constant multiple of either $xy$ or $yz$. To account for this, we have:

\begin{prop}\label{prop:theDecomp}
Let $A$ be a strongly Koszul algebra and $J$ a monomial ideal in $A$. For any $v \in M(J)$, there exists a decomposition
$$v = \sum_i m_i^* (v) \cdot m_i,$$
where, if $m_i^* (v) \neq 0$, then $m_i^* (v) \cdot m_i \notin J_{i-1}$ for all $1 \leq i \leq n$.
\end{prop}

\begin{proof}
The fact that such a decomposition exists is clear, since if $m_i^* (v) \cdot m_i \in J_{i-1}$ then we can rewrite $m_i^* (v) m_i$ in terms of the generators of $J_{i-1}$ and rechoose $m_i^* (v)$ to be $0$.
\end{proof}

By convention, if $m_j \in G(J)$ and $x_s m_j \notin J_{j-1}$, then we will set $m_i (x_s m_j) = 0$ for $i < j$ and $m_j (x_s m_j) = x_s$.




\begin{notation}\label{not:Lideals}
Let $J = (m_1 , \dots , m_r)$ be an ideal admitting linear quotients with the given ordering. For each $i \geq 1$, write $(J_{i-1} : m_i) = (x_j \mid j \in E_i)$, where $E_i \subset [n]$. Define 
$$B^i := A^! / L^i,$$
where $L^i := \sum_{j \notin E_i} A^! x_j^*$.
\end{notation}

Finally, we arrive at the analogue of iterated mapping cones for strongly Koszul algebras:

\begin{lemma}\label{lem:IMCbasis}
Let $A$ be a strongly Koszul algebra and $J = (m_1 , \dots , m_r)$ a monomial ideal admitting linear quotients with respect to the given order, where $\deg (m_1) \leq \cdots \leq \deg(m_r)$. Then the iterated mapping cone derived from the sequence $m_1 , \dots , m_r$ is a minimal graded free resolution of $A/J$ with basis elements in homological degree $\ell \geq 1$ of the form
$$\{ m_i \otimes f \mid 1 \leq i \leq r, \ f \ \textrm{is a basis element of} \ \big( B^i_{\ell-1} \big)^* \}.$$
\end{lemma}

\begin{proof}
The proof is by induction on the minimal number of generators of $J$, denoted $r$. If $r = 1$, then by Observation \ref{obs:varsColon} and Proposition \ref{prop:varslinRes}, the quotient ring $A / (0:m)$ has minimal free resolution obtained by Construction \ref{cons:subPriddy}. Thus, let $F_\bullet$ denote the minimal free resolution of $A/(0:m)$, as in Construction \ref{cons:subPriddy}. By the tautological short exact sequence
$$0 \to \frac{A}{(0:m)} \xrightarrow[]{\cdot m} A \to \frac{A}{(m)} \to 0,$$
one immediately finds that the minimal free resolution of $A/(m)$ is obtained as the augmentation of $F_\bullet$ by the homothety $A \xrightarrow[]{\cdot m} A$. This concludes the base case of the induction.

Let $r>1$ and $J = (m_1 , \dots , m_r)$. To conclude the proof, apply the inductive hypothesis to the short exact sequence
$$ 0 \to \frac{A}{(m_1 , \dots , m_{r-1} : m_r )} \xrightarrow[]{\cdot m_r} \frac{A}{(m_1 , \dots , m_{r-1})} \to \frac{A}{(m_1 , \dots , m_r)} \to 0.$$
The mapping cone associated to the above short exact sequence will be minimal by Proposition \ref{prop:varslinRes} combined with that assumption that the degrees of the minimal generators are increasing.
\end{proof}

\begin{cor}
Let $A$ be a strongly Koszul algebra and $J = (m_1 , \dots , m_r)$ a monomial ideal admitting linear quotients with respect to the given order, where $\deg (m_1) \leq \cdots \leq \deg(m_r)$. Then
$$\beta_\ell (A/J) = \sum_{m_i \in G(J)} \rank B_{\ell-1}^i,$$
$$\beta_{\ell, \ell+q} (A/J) =   \sum_{\substack{m_i \in G(J), \\
\deg (m_i) = q}} \rank B_{\ell-1}^i, \quad \textrm{and}$$
$$\operatorname{reg} (J) = \max \{ \deg (m_i ) \mid m_i \in G(J) \}.$$
In particular, $J$ has a linear minimal free resolution if and only if $J$ is equigenerated.
\end{cor}

To conclude this section, we end with an example applying some of our techniques to a case considered in \cite[Example 6.4]{faber2020canonical}. It is worth mentioning that, although we obtain the Betti number formula more easily in this case, the construction of \cite{faber2020canonical} gives the minimal free resolution explicitly.

\begin{example}
Let $A := \frac{k[x_1 , \dots , x_n]}{(x_1^2 , \dots , x_n^2)}$, so that $A^! = \frac{k\langle x_1^* , \dots  ,x_n^* \rangle}{(x_i^* x_j^* + x_j^* x_i^* \mid i <j )}$. Consider the ideal $\m^d$, where $d \leq n$; this ideal is minimally generated by all squarefree monomials of degree $d$, and hence can be thought of as parametrized by all subsets $\sigma \subset [n]$ of size $d$. Ordering the generators of $\m^d$ lexicographically, one finds
$$(\m^d)_{< \sigma} : x_\sigma = (x_i \mid i \leq \max (\sigma) ),$$
whence $\m^d$ admits linear quotient with respect to the lexicographic ordering. Moreover,
$$\rank B_{i}^\sigma = \binom{i + \max(\sigma) -1}{\max (\sigma) - 1},$$
where $B_i^\sigma$ denotes the degree $i$ component of the quotient of $A^!$ defined by the left ideal $\sum_{i= \max(\sigma) +1}^n A^! x_i^*$. One then computes:
\begingroup\allowdisplaybreaks
\begin{align*}
    \beta_{i,i+d} (\m^d ) &= \sum_{j=d}^n \sum_{\substack{\sigma \subset [j-1] \\
    |\sigma| = d-1 \\}} \binom{i+j-1}{j-1} \\
    &= \sum_{j=d}^n \binom{j-1}{d-1} \binom{i+j-1}{j-1} \\
    &= \frac{n-d+1}{d+i} \binom{n}{d-1} \binom{n+i}{n}. \\
\end{align*}
\endgroup
This recovers a simpler form of the formula computed in \cite{faber2020canonical}.
\end{example}

\section{Minimal Free Resolutions for Ideals Admitting Regular Orderings}\label{sec:theMFRforRegular}

In this section, we consider the ``correct" generalization of a regular decomposition function for ideals admitting linear quotients in a polynomial ring. It turns out that this condition, which we refer to as a \emph{regular ordering}, is a little bit more subtle to state in the case of strongly Koszul algebras. However, once formulated, we obtain the direct analogue of the explicit differentials obtained in Theorem $1.12$ of \cite{herzog2002resolutions}; see Theorem \ref{thm:theMFR}. The proof of this theorem is analogous to the original proof, with the computations being slightly more delicate. We conclude with questions about how far the original material of iterated mapping cones can be generalized for strongly Koszul algebras.

To begin this section, we recall the original definitions given for (regular) decomposition functions, as given in \cite{herzog2002resolutions}.

\begin{definition}\label{def:decompFunc}
Let $A$ be a standard graded polynomial ring and $J = (m_1 , \dots , m_r)$ a monomial ideal of $A$. The \emph{decomposition function} $g : M(J) \to G(J)$ is the function defined by
$$g(u) = m_j,$$
where $1 \leq j \leq r$ is the smallest index such that $u \in J_j$. For any $1 \leq i \leq r$, define
$$\set (m_i) := \{ k \in [n] \mid x_k m_i \in J_{i-1} \}.$$
\end{definition}

\begin{definition}\label{def:regularDecomp}
Let $g : M(J) \to G(J)$ be a decomposition function. Then $g$ is \emph{regular} if $\set (g(x_s u)) \subset \set (u)$ for all $s \in \set (u)$ and $u \in G(J)$. 
\end{definition}

For the more general formulation, it turns out that the relations imposed by the quadratic ideal defining $A$ will be needed. We adopt the following:

\begin{setup}\label{set:strongKoszulRegular}
Let $A = k[x_1 , \dots , x_n] /I$ be a strongly Koszul algebra. Let $S$ be a set of pairs $(s,t) \in S$ such that $x_s x_t$ is a basis element of $A_2$ and there exist coefficients $f^{u,v}_{s,t} \in k$ for all $(u,v) \notin S$ such that
$$x_u x_v = \sum_{(s,t) \in S} f^{u,v}_{s,t} x_s x_t.$$
Recall that the relations on $A^!_2$ are given by Observation \ref{obs:A!relations}.
\end{setup}

In the following definition, recall that the definition of $L^i$ is given in Notation \ref{not:Lideals}.

\begin{definition}\label{def:regulrityDef}
Adopt notation and hypotheses as in Setup \ref{set:strongKoszulRegular}. A monomial ideal $J = (m_1 , \dots , m_r)$ admits a \emph{regular ordering} if:
\begin{enumerate}
    \item For all $j \leq k$, 
    $$x_u m_j^* (x_v m_k) + x_v m_j^* (x_u m_k) = \sum_{(s,t) \in S} f_{s,t}^{u,v} \big( x_s m_j^* (x_t m_k) + x_t m_j^* (x_s m_j) \big)$$
    \item For all $j < k$ and $x_t^* \notin L^k$, if $m_j^* (x_t m_k ) \neq 0$, then: 
    \begin{enumerate}
        \item $L^k \subset L^j$, and
        \item if $x_t^* L^j \not\subset L^k$ and $x_t^* x_s^* L^j \subset L^k$ for some $s$, then $x_s^* \in L^k$.
    \end{enumerate}
    \item for all $i < k$ and $x_s, \ x_t \in (J_{k-1} : m_k)$, the following equality holds:
$$m_i^* (x_s x_t m_k ) =\sum_{i \leq j < k} m_i^* (x_s m_j ) m_j^* (x_t m_k ).$$
\end{enumerate}
\end{definition}

Some comments on the intuition of conditions $(1)-(3)$ in the above proof are in order. The condition $(1)$ is a compatibility condition between the decomposition functions $m_i^*(-)$ and the relations defining the Koszul algebra. We shall see that in the case of a polynomial ring, $(1)$ is trivially satisfied, and hence only plays a role in more general Koszul algebras. The condition $(2)(a)$ is a translation of the regularity condition in Definition \ref{def:regularDecomp} for the functions $m_i^*$, and $(2)(b)$ is a condition that we will see implies well-definedness of the generalized iterated mapping cone construction. Finally, condition $(3)$ in the case of a polynomial ring reduces to the statement that $g(x_s x_t m_k) = g(x_s g(x_t m_k))$ for every $1 \leq k \leq r$, and so we see that $(2)(a)$ implies $(3)$ when $A$ is a polynomial ring. 



As a quick sanity check, we should probably check that Definition \ref{def:regulrityDef} is actually a generalization of Definition \ref{def:regularDecomp}. Indeed:

\begin{obs}
Let $A$ be a standard graded polynomial ring and $J$ a monomial ideal with linear quotients admitting a regular decomposition function $g : M(J) 
\to G(J)$. Then $J$ admits a regular ordering.
\end{obs}

\begin{proof}
Notice that condition $(1)$ of Definition \ref{def:regulrityDef} in the case of a commutative polynomial ring reduces to:
$$x_u m_j^* (x_v m_k) + x_v m_j^* (x_u m_k) = x_v m_j^* (x_u m_k) + x_u m_j^* (x_v m_k)$$
for any $u > v$. Thus, the condition $(1)$ holds trivially. It is then a quick exercise to verify that condition $(2)(a)$ is a simple translation of Definition \ref{def:regularDecomp}. For $(2)(b)$, notice that if $x_t^* L^j \not\subset L^k$, then there exists some $x_u^* \in L^j$ that is not contained in $L^k$. Since each $L^j$ and $L^k$ are ideals in an exterior algebra, the element $x_t^* x_s^* x_u^* = - x_t^* x_u^* x_s^*$ is contained in $L^k$ if and only if $x_s^* \in L^k$.

Finally, as previously mentioned, condition $(3)$ is a retranslation of the fact that $g(x_s x_t u) = g(x_s g(x_tu))$ for regular decomposition functions.
\end{proof}

Observe that if $A$ is a multigraded Koszul algebra (that is, $A$ is the quotient by a quadratic monomial ideal), then $A$ is strongly Koszul and the decomposition $g$ of Definition \ref{def:decompFunc} is still well-defined. However, in this case a regular decomposition function is \emph{not} enough to guarantee that the ideal admits a regular ordering.

\begin{prop}
Let $A = k[x_1 , \dots , x_n ] /I$ be a multigraded Koszul algebra and $J = (m_1 , \dots , m_r)$ be a monomial ideal with linear quotients admitting a regular decomposition function. Then $J$ admits a regular ordering if:
\begin{enumerate}[(*)]
    \item for all $1 \leq i \leq r$ and all $y \in \supp (I)$, either $y m_i \notin J_{i-1}$ or $y m_i = 0$.
\end{enumerate}
\end{prop}

\begin{proof}
The proof is broken up in line with each condition in Definition \ref{def:regulrityDef}. \\
\textbf{Proof of (1):} Condition $(1)$ of Definition \ref{def:regulrityDef} in the case that $A$ is a multigraded Koszul algebra is implied by the following:
$$x_u m_j^* ( x_v m_k) = 0 \ \textrm{for all} \ x_u x_v \in I \ \textrm{and} \  1 \leq j \leq k \leq r.$$
Assume that condition $(*)$ holds. If $x_v m_k = 0$, then evidently $x_u m_j^* (x_v m_k) = 0$, so assume that $x_v m_k \notin J_{k-1}$. By definition, one has $m_j^* (x_v m_k) = \delta_{jk} x_v$, whence $x_u m_j^* (x_v m_k) = \delta_{jk} x_u x_v = 0$. \\
\textbf{Proof of (2)(a)-(b):} Observe first that if $y \notin \supp (I)$, then $y^*$ anticommutes with every other variable in $A^!$. Thus the proof of $(2)(a)$ and $(2)(b)$ is identical to the polynomial ring case when $y \notin \supp (I)$. If $y \in \supp (I)$, then condition $(*)$ implies that $m_j^* (y m_k) = 0$ for all $j<k$, so there is nothing to check. \\
\textbf{Proof of (3):} Let $x_s, x_t \in (J_{k-1} : m_k)$ for some $1 \leq k \leq r$. There are three cases to check: \\
\textbf{Case 1:} Neither $x_s$ nor $x_t \in \supp (I)$. In this case, notice that $x_s x_t m_k \neq 0$ since otherwise $m_k = 0$. The proof then reduces to the same proof given in \cite[Lemma 1.11]{herzog2002resolutions}. \\
\textbf{Case 2:} $x_t \in \supp (I)$. Since $x_t \in (J_{k-1} : m_k)$ by assumption, condition $(*)$ forces $x_t m_k = 0$. Thus both sides of the equality in $(3)$ are $0$. \\
\textbf{Case 3:} $x_s \in \supp (I)$ and $x_t \notin \supp (I)$. It suffices to show that $x_s g( x_t m_k) = 0$. However, if $x_s g(x_t m_k) \neq 0$, then $(*)$ implies that $x_s \notin (J_{j-1} : m_j)$, where $m_j := g(x_t m_k)$ and $j<k$. Definition \ref{def:regularDecomp} then implies that $x_s \notin (J_{k-1} : m_k)$, a clear contradiction. 
\end{proof}

\begin{example}
Let $J := (x_1 x_2 , x_2 x_3)$ in the Koszul algebra $k[x_1,x_2,x_3] /(x_1x_3, x_3^2)$. Then with the given ordering, $J$ has linear quotients and admits a regular ordering.
\end{example}

The following lemma will be essential for showing that the differentials appearing in Theorem \ref{thm:theMFR} are well-defined.

\begin{lemma}\label{lem:wellDefComp}
Adopt notation and hypotheses as in Setup \ref{set:strongKoszulRegular} and assume that the monomial ideal $J = (m_1 , \dots , m_r)$ admits a regular ordering. Let $f \in (B_\ell^k)^*$ and assume that for any $j$ with $m_j^* (x_t m_k) \neq 0$, one has $x_t^* L^j \not\subset L^k$. Then the following statements hold:
\begin{enumerate}
    \item If $x_t^* x_s^* L^j \not\subset L^k$ for some $s$, then $x_t^* x_s^* L^i \not\subset L^k$ for every $i < j$ such that $m_i^* (x_s m_j) \neq 0$.
    \item If $x_t^* x_s^* L^j \subset L^k$ for some $s$, then $x_s^* \in L^j$. In particular, 
    $$m_i^* (x_s m_j) = \begin{cases} 
x_s & \textrm{if} \ i=j, \\
0 & \textrm{otherwise}.
\end{cases}$$
\end{enumerate}
\end{lemma}

\begin{proof}
\textbf{Proof of (1):} Suppose for sake of contradiction that there exists some $i<j$ with $m_i^* (x_s m_j) \neq 0$, but $x_t^* x_s^* L^i \subset L^k$. Then, by $(2)(a)$ of Definition \ref{def:regulrityDef}, one has that $L^j \subset  L^i$, implying
$$x_t^* x_s^* L^j \subset x_t^* x_s^* L^i \subset L^k,$$
which is a contradiction. \\
\textbf{Proof of (2):} A priori, condition $(2)(b)$ of Definition \ref{def:regulrityDef} implies that $x_s^* \in L^k$. But $(2)(a)$ of Definition \ref{def:regulrityDef} implies that $L^k \subset L^j$, so $x_s^* \in L^j$. This is equivalent to saying $x_s m_j \notin J_{j-1}$, whence the latter statement follows immediately. 
\end{proof}

We now arrive at the main result of this section. As previously mentioned, this gives the generalized version of the explicit minimal free resolution of ideals with regular decomposition functions for strongly Koszul algebras. 

\begin{theorem}\label{thm:theMFR}
Adopt notation and hypotheses as in Setup \ref{set:strongKoszulRegular}. Let $J = (m_1 , \dots , m_r)$ be a monomial ideal with linear quotients admitting a regular ordering, where $\deg (m_1) \leq \cdots \leq \deg(m_r)$. Then the minimal free resolution $F_\bullet$ of $A/J$ has differentials of the form
$$\partial (m_k \otimes f) = - \sum_s x_s (m_k \otimes  f x_s^*) + \sum_{s; \ j \leq k} m^*_j (x_s m_k)  ( m_j \otimes  f x_s^*), \quad \textrm{for} \ \deg(f) >0,$$
$$\partial (m_k \otimes 1) = m_k,$$
where each $m_i \otimes f$ is as in the statement of Lemma \ref{lem:IMCbasis}.
\end{theorem}

\begin{proof}
Throughout the proof, right multiplication will be used to denote the opposite action of the left action. More precisely: $x_t^* f := f x_t^*$ (so that $x_s^* x_t^* f = f x_t^* x_s^*$). The proof follows by induction on $r$. When $r=1$, the ideal $J = (m)$ is principal and the result follows from the base case of the proof of Lemma \ref{lem:IMCbasis}. Assume now that $r>1$ and consider the short exact sequence:
$$ 0 \to \frac{A}{(J_{r-1} : m_r )} \xrightarrow[]{\cdot m_r} \frac{R}{J_{r-1}} \to \frac{R}{J} \to 0.$$
By the inductive hypothesis, the differentials of the minimal free resolution $F_\bullet$ of $R/ J_{r-1}$ are as in the statement of the Theorem; let $K_\bullet$ denote the minimal free resolution of $R/ (J_{r-1} : m_r)$ provided by Construction \ref{cons:subPriddy}. The proof follows by showing that the following map is a morphism of complexes:
\begingroup\allowdisplaybreaks
\begin{align*}
    \psi_i : K_i &\to F_i \\
    m_k \otimes f &\mapsto \sum_{t; \ j<k} m_j^* (x_t m_k ) (m_j \otimes x_t^* f ). \\ 
\end{align*}
\endgroup
That is, the following diagram commutes:
\begin{equation}\label{eq:theDiagram}
    \xymatrix{K_i \ar[r]^-{\psi_i} \ar[d] & F_i \ar[d] \\
K_{i-1} \ar[r]^-{\psi_{i-1}} & F_{i-1}. \\}
\end{equation}
For $i=1$ this is evident, so assume that $i \geq 2$. Moving clockwise around diagram (\ref{eq:theDiagram}), one obtains:
\begingroup\allowdisplaybreaks
\begin{align*}
    m_k \otimes f &\mapsto \sum_{t; \ j<k} m_j^* (x_t m_k ) (m_j \otimes x_t^* f) \\
    &\mapsto -\sum_{s} \sum_{t; \ j<k} x_s m_j^* (x_t m_k ) (m_j \otimes  x_s^* x_t^* f) \\
    &+ \sum_{t, \ j<k}\sum_{s ; \ i \leq j} m_i^* (x_s m_j) m_j^* (x_t m_k ) (m_i \otimes x_s^* x_t^* f). \\
\end{align*}
\endgroup
Observe that
\begingroup\allowdisplaybreaks
\begin{align*}
    \sum_{s} \sum_{t; \ j<k} x_s m_j^* (x_t m_k ) (m_j \otimes  x_s^* x_t^* f) &= \sum_{(s,t) \in S} \sum_{j < k} x_s m_j^* (x_t m_k ) (m_j \otimes  x_s^* x_t^* f) \\
    &+ \sum_{(u,v) \notin S} \sum_{j<k} x_u m_j^* (x_v m_k ) (m_j \otimes  x_u^* x_v^* f) \\
    =  \sum_{(u,v) \notin S} \sum_{j<k}& \Big( x_u m_j^* (x_v m_k ) - \sum_{(s,t) \in S} f^{u,v}_{s,t} x_s m_j^*(x_t m_k) \Big) (m_j \otimes  x_u^* x_v^* f) \\
    = -\sum_{(u,v) \notin S} \sum_{j<k}& \Big( x_v m_j^* (x_u m_k ) - \sum_{(s,t) \in S} f^{u,v}_{s,t} x_t m_j^*(x_s m_k) \Big) (m_j \otimes  x_u^* x_v^* f), \\
\end{align*}
\endgroup
where the final equality follows by condition $(1)$ of Definition \ref{def:regulrityDef}. Next, recall the notation $L^i$ for $1 \leq i \leq r$ from Notation \ref{not:Lideals}. It still remains to show that this differential is well defined; that is, if $m_j^* (x_t m_k) \neq 0$ and $m_j \otimes x_t^*f$ is not a valid basis element (ie, $fx_t^* (L^j) \neq 0$), then the image of this basis element under the differential should remain $0$. By definition, notice that $f x_t^* (L^j) \neq 0 \iff x_t^* L^j \not\subset L^k$. There are two cases to consider:

\textbf{Case 1:} $m_j \otimes x_s^* x_t^* f$ is \emph{not} a valid basis element. This is equivalent to saying $x_t^* x_s^* L^j \not\subset L^k$. By $(1)$ of Lemma \ref{lem:wellDefComp}, it follows that $m_i \otimes f x_t^* x_s^*$ is not a valid basis element for any $i$ with $m_i^* (x_s m_j) \neq 0$.

\textbf{Case 2:} $m_j \otimes x_s^* x_t^* f$ is a valid basis element. This is equivalent to saying $x_t^* x_s^* L^j \subset L^k$, and by $(2)$ of Lemma \ref{lem:wellDefComp}, the following equality holds: 
$$m_i^* (x_s m_j) = \begin{cases} 
x_s & \textrm{if} \ i=j, \\
0 & \textrm{otherwise}.
\end{cases}$$

One then computes:
$$m_j \otimes x_t^*f \mapsto -\sum_{s} x_s (m_j \otimes x_s^* x_t^* f) + \sum_{i \leq j} m_i^* (x_s m_j) (m_i \otimes x_s^* x_t^* f).$$ 
Combining both of the above cases, the expression on the right must also be $0$. Finally, it remains to consider the following:
\begingroup\allowdisplaybreaks
\begin{align*}
    \sum_{t; \ j<k}\sum_{s; \ i \leq j} m_i^* (x_s m_j) m_j^* (x_t m_k ) (m_i \otimes x_s^* x_t^* f) & = \sum_{(s,t) \in S} \sum_{i \leq j < k} m_i^* (x_s m_j) m_j^* (x_t m_k ) (m_i \otimes x_s^* x_t^* f) \\
    &+ \sum_{(u,v) \notin S} \sum_{i \leq j<k} m_i^* (x_u m_j) m_j^* (x_v m_k ) (m_i \otimes x_u^* x_v^* f). 
\end{align*}
\endgroup
Focusing on the bottom term, one computes:
\begingroup\allowdisplaybreaks
\begin{align*}
    &\sum_{(u,v) \notin S} \sum_{i \leq j<k} m_i^* (x_u m_j) m_j^* (x_v m_k ) (m_i \otimes x_u^* x_v^* f) \\
    =& \sum_{(u,v) \notin S} \sum_{i<k} m_i^* (x_u x_v m_k) (m_i \otimes x_u^* x_v^* f) \quad (\textrm{by (3) in Definition \ref{def:regulrityDef}})\\
    = & \sum_{(u,v) \notin S} \sum_{i < k} \sum_{(s,t) \in S} f_{s,t}^{u,v} m_i^* (x_s x_t m_k) (m_i \otimes x_u^* x_v^* f) \\
    =& \sum_{(s,t) \in S)} \sum_{i < k} m_i^* (x_s x_t m_k ) (m_i \otimes \sum_{(u,v) \notin S} f^{u,v}_{s,t} x_u^* x_v^* f) \\
    = & -\sum_{(s,t) \in S)} \sum_{i < k} m_i^* (x_s x_t m_k ) (m_i \otimes x_s^* x_t^* f ). \\
\end{align*}
\endgroup
Thus, the term $\sum_{t, \ j<k}\sum_{s , \ i \leq j} m_i^* (x_s m_j) m_j^* (x_t m_k ) (m_i \otimes x_s^* x_t^* f)$ is identically $0$. 

Moving counterclockwise around diagram (\ref{eq:theDiagram}), one has:
\begingroup\allowdisplaybreaks
\begin{align*}
    m_k \otimes f &\mapsto \sum_t x_t (m_k \otimes x_t^* f) \\
    &\mapsto \sum_t \sum_{s , \ j < k} x_t m_j^* (x_s m_k ) (m_j \otimes x_s^* x_t^* f) \\
    &= \sum_{(s,t) \in S} \sum_{j<k} x_t m_j^* (x_s m_k ) (m_j \otimes x_s^* x_t^* f) \\
    &+ \sum_{(u,v) \notin S} \sum_{j<k} x_v m_j^* (x_u m_k ) (m_j \otimes x_u^* x_v^*f ) \\
    = \sum_{(u,v) \notin S} \sum_{j<k}& \Big( x_v m_j^* (x_u m_k ) - \sum_{(s,t) \in S} f^{u,v}_{s,t} x_t m_j^*(x_s m_k) \Big) (m_j \otimes  x_u^* x_v^* f). \\
\end{align*}
\endgroup
Comparing this with the expression obtained after moving clockwise, one finds that these terms are equal, whence the result.
\end{proof}

Finally, we conclude this section with some discussion about further results for ideals admitting linear quotients in strongly Koszul algebras. Firstly, it is worth noting that the original definition of decomposition function given in \cite{herzog2002resolutions} is not the only choice that one can make. For instance, for \emph{cointerval ideals} (that is, edge ideals corresponding to cointerval graphs), one can formulate a different definition of a decomposition function to construct the morphism of complexes appearing in the iterated mapping cone construction; see Section $4$ of work by Dochtermann and Mohammadi \cite{dochtermann2014cellular}. The choice of a different decomposition function is, in our notation, a different choice of the unique decomposition of Proposition \ref{prop:theDecomp}. This then begs the question:

\begin{question}
What other choices of decompositions as in Proposition \ref{prop:theDecomp} can be used to construct a morphism of complexes as in the proof of Theorem \ref{thm:theMFR}?
\end{question}

Continuing with the work of \cite{dochtermann2014cellular}, we can also ask whether it makes sense to try to put a cell complex structure on the resolution of Theorem \ref{thm:theMFR}. This question seems to be less well-posed, since we have the obvious initial problem of needing an infinite-dimensional cell complex for general strongly Koszul algebras. Moreover, the differentials of the Priddy complex are more difficult to deal with than the classical Koszul complex, so it seems unlikely that the resolution of Theorem \ref{thm:theMFR} will be cellular in any sort of generality. 

\bibliographystyle{amsplain}
\bibliography{biblio}
\addcontentsline{toc}{section}{Bibliography}

\end{document}